\documentclass[a4paper]{amsart}

\usepackage{amsfonts}
\usepackage{verbatim}
\usepackage{amssymb}
\usepackage{amscd}
\usepackage{amsmath}
\usepackage{latexsym}
\usepackage{framed, color}

\numberwithin{equation}{section}
\usepackage[dvips]{graphicx}
\usepackage{psfrag}

\newtheorem{thm}{Theorem}[section]
\newtheorem{cor}[thm]{Corollary}
\newtheorem{prop}[thm]{Proposition}
\newtheorem{lem}[thm]{Lemma}

\theoremstyle{definition}
\newtheorem{dfn}[thm]{Definition}

\theoremstyle{remark}
\newtheorem{rem}[thm]{Remark}
\newtheorem{problem}[thm]{Problem}
\newtheorem{ex}[thm]{Example}

\makeatletter
 
 \@addtoreset{equation}{section}
\makeatother

\newcommand{\R}{\mathbb{R}}

\newcommand{\Q}{\mathbb{Q}}

\newcommand{\Z}{\mathbb{Z}}

\newcommand{\Int}{\operatorname{Int}}
\newcommand{\id}{\operatorname{id}}

\renewcommand{\tilde}{\widetilde}

\def\spmapright#1{\smash{%
 \mathop{\hbox to 1.3cm{\rightarrowfill}}
  \limits^{#1}}}
\def\spmapleft#1{\smash{%
 \mathop{\hbox to 1.3cm{\leftarrowfill}}
  \limits^{#1}}}

\title[Round fold maps of $n$--dimensional manifolds into $\R^{n-1}$]
{Round fold maps of $n$--dimensional manifolds \\ into $\R^{n-1}$}
\author{Naoki Kitazawa and Osamu Saeki} 
\address{Institute of Mathematics for Industry,
Kyushu University,
Motooka 744, Nishi-ku, Fukuoka 819-0395, Japan}

\email{n-kitazawa@imi.kyushu-u.ac.jp}

\email{saeki@imi.kyushu-u.ac.jp}
\date{\today}
\keywords{round fold map, simple stable map, diffeomorphisms of surfaces, Morse function}

\subjclass[2010]{Primary 57R45;
Secondary 
58K30}

\begin{document}
\begin{abstract}
We determine those smooth $n$--dimensional closed manifolds
with $n \geq 4$
which admit round fold maps into $\R^{n-1}$, i.e.\ fold
maps whose critical value sets consist of disjoint
spheres of dimension $n-2$ isotopic to concentric spheres.
We also classify such round fold maps up to $C^\infty$
$\mathcal{A}$--equivalence.
\end{abstract}

\maketitle

\section{Introduction}\label{section1}

Let $M$ be a smooth closed manifold of dimension
$n \geq 2$. A smooth map $f : M \to \R^p$ with $n \geq p \geq 1$
is called a \emph{fold map} if it has only fold points as its
singularities (for details, see \S\ref{section2}).
Note that fold points are the simplest
singularities among those which appear generically \cite{GG}
and that fold maps are natural generalizations of Morse functions.

In \cite{Sa92, Sa93-2, Sa2, Sa1}, the second author
considered the following smaller class
of generic smooth maps. A fold map $f : M \to \R^p$
is {\em simple} if for every $q \in \R^p$, each component 
of $f^{-1}(q)$ contains at most one singular point.
In particular, if $f|_{S(f)}$ is an embedding,
then $f$ is simple, where $S(f) (\subset M)$ denotes the set of
singular points of $f$. Note that if $f$ is a fold
map, then $S(f)$ is a regular closed submanifold of $M$ of dimension $p-1$
and that $f|_{S(f)}$ is an immersion in general.
In \cite{Sa1},
the second author proved that a closed orientable $3$--dimensional
manifold $M$ admits a fold
map $f : M \to \R^2$ such that $f|_{S(f)}$ is an embedding
if and only if $M$ is a graph manifold, where a closed
orientable $3$--dimensional manifold is a graph manifold
if it is the finite union of $S^1$--bundles over compact surfaces
attached along their torus boundaries.
Thus, for example, if $M$ is hyperbolic, then
$M$ never admits such a fold map, although every closed
orientable $3$--dimensional manifold admits a fold map into $\R^2$
by \cite{L0}.

On the other hand, the first author introduced the notion
of a round fold map \cite{Ki1, Ki2}: a smooth
map $f : M \to \R^p$ is a \emph{round fold map}
if it is a fold map and $f|_{S(f)}$ is an embedding
onto the disjoint union of some concentric $(p-1)$--dimensional spheres
in $\R^p$ (for details,
see \S\ref{section2}). Round fold maps are naturally
simple. As has been studied by the
first author, round fold maps have various nice properties.

The main result of this paper is Theorem~\ref{thm:main}, which
characterizes those smooth closed $n$--dimensional
manifolds that admit round fold maps into $\R^{n-1}$ for
$n \geq 4$.
We also classify such round fold maps up to $C^\infty$ 
$\mathcal{A}$--equivalence (see Theorems~\ref{thm:classification}
and \ref{thm:classification2}).

The paper is organized as follows. In \S\ref{section2},
we prepare several definitions concerning
round fold maps together with some examples,
and state
our main theorem. 
In \S\ref{section3}, we prove the
main theorem mentioned above. The main ingredients
are the celebrated results about the homotopy
groups of diffeomorphism groups of compact surfaces 
\cite{EE1, EE2, ES}. We also give
a similar characterization of those smooth
orientable closed connected $n$--dimensional manifolds
that admit directed round fold maps into $\R^{n-1}$,
where a fold map is \emph{directed} if the number of regular fiber
components increases toward the central region of $\R^{n-1}$
(for details, see \S\ref{section3}).
In \S\ref{section4}, we give some related results
and remarks. Finally in \S\ref{section5}, we classify
the round fold maps as in Theorem~\ref{thm:main}
up to $C^\infty$ $\mathcal{A}$--equivalence using
Morse functions on compact surfaces.
The key idea for the classification 
is to use some results about the homotopy
type of the group of diffeomorphisms of compact
surfaces that preserve a given Morse function,
due to Maksymenko \cite{Mak1, Mak2}.

Throughout the paper, 
all manifolds and maps between them are smooth
of class $C^\infty$ unless otherwise specified. 
For a space $X$, 
$\id_X$ denotes the identity map of $X$. 
The symbol ``$\cong$'' denotes a diffeomorphism between
smooth manifolds.

\section{Round fold maps}\label{section2}

Let $M$ be a closed $n$--dimensional manifold and
$f : M \to \R^p$ a smooth map, where we assume
$n \geq p \geq 2$. By the \emph{codimension} of
$f$ we mean the integer $p-n \leq 0$.

\begin{dfn}
A point $q \in M$ is a \emph{singular point} of $f$
if the rank of the differential $df_q : T_qM \to T_{f(q)}\R^p$
is strictly smaller than $p$. We denote by $S(f)$ the set
of all singular points of $f$.
A point $q \in S(f)$ is a \emph{fold point}
if $f$ is represented by the map
$$(x_1, x_2, \ldots, x_n) \mapsto (x_1, x_2, \ldots, x_{p-1},
\pm x_p^2 \pm x_{p+1}^2 \pm \cdots \pm x_n^2)$$
around the origin with respect to certain local coordinates around $q$
and $f(q)$. Let $\lambda$ be the number of negative signs
appearing in the above expression. The integer
$$\max\{\lambda, n-p+1 - \lambda\} \in \{\lceil (n-p+1)/2 \rceil,
\lceil (n-p+1)/2 \rceil+1, \ldots, n-p+1\}$$
is called the \emph{absolute index} of the fold point $q$,
which is known to be well-defined, where for $x \in \R$,
$\lceil x \rceil$ denotes the minimum integer greater than or equal to $x$.
We call a point $q \in S(f)$ a \emph{definite fold point}
if its absolute index is equal to $n-p+1$, otherwise an
\emph{indefinite fold point}.

A smooth map $f : M \to \R^p$ is called a \emph{fold map}
if it has only fold points as its singular points.
Note that then $S(f)$ is a closed $(p-1)$--dimensional submanifold of $M$
and that $f|_{S(f)}$ is an immersion.
\end{dfn}

Note that if $p=n-1$, then the absolute index of a fold
point is equal either to $1$ or to $2$.

\begin{dfn}
Let $C$ be a finite disjoint union of embedded $(p-1)$--dimensional
spheres in $\R^p$, $p \geq 2$.
We say that $C$ is \emph{concentric} if each component
bounds a $p$--dimensional disk in $\R^p$ and
for every pair $c_0, c_1$
of distinct components of $C$, exactly one of them, say $c_i$, is contained
in the bounded region of $\R^p \setminus c_{1-i}$ (see Fig.~\ref{fig1}
for $p = 2$).
(In this case, we say that $c_i$ (or $c_{1-i}$)
is an \emph{inner component} (resp.\ an \emph{outer component})
with respect to $c_{1-i}$ (resp.\ $c_i$).)
In other words, $C$ is isotopic to a set of concentric $(p-1)$--dimensional
spheres in $\R^p$. (Note that the condition that each component
of $C$ bounds a $p$--dimensional disk is redundant for $p \neq 4$.)
\end{dfn}

\begin{dfn}
We say that a smooth map $f : M \to \R^p$ of a closed $n$--dimensional
manifold $M$ into the $p$--dimensional Euclidean space 
is a \emph{round fold map}
if it is a fold map and $f|_{S(f)}$ is an embedding onto a concentric
family of embedded spheres.
Note that a round fold map is a simple stable map in the sense of
\cite{Sa92, Sa93-2, Sa2, Sa1}. Note also that the outermost component
of $f(S(f))$ consists of the images of definite fold points.
\end{dfn}

\begin{figure}[h]
\centering
\psfrag{C}{$C$}
\psfrag{ci}{$c_i$}
\psfrag{c1}{$c_{1-i}$}
\includegraphics[width=0.9\linewidth,height=0.3\textheight,
keepaspectratio]{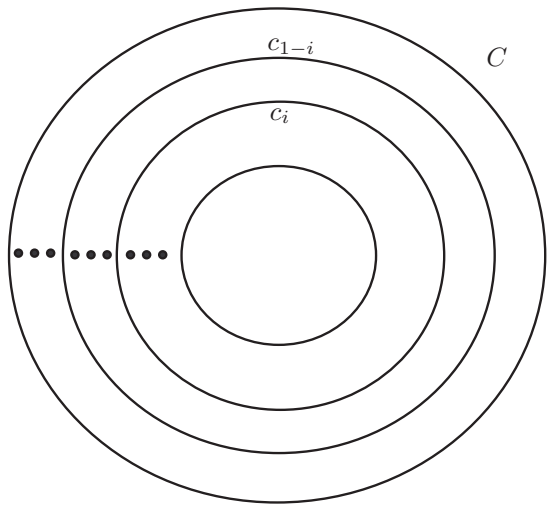}
\caption{Family of concentric $1$--dimensional spheres in $\R^2$}
\label{fig1}
\end{figure}

\begin{ex}\label{ex1}
Let $F$ be a compact connected $m$--dimensional manifold
possibly with boundary and
$h : F \to [1/2, \infty)$ a Morse function
such that $h(\partial F) = 1/2$ and that $h$ has no
critical point near the boundary.
(Throughout the paper, a Morse function is a smooth
function whose critical points are all non-degenerate and
have distinct critical values.)
Then for $p \geq 2$ we can construct a round fold map
$f : M \to \R^p$ in such a way that $M$ is the closed
$(m+p-1)$--dimensional manifold
$(\partial F \times D^p) \cup (F \times \partial D^p)
= \partial (F \times D^p)$,
that $f$ restricted to $F \times \{x\}$ can be identified with the
Morse function $h$ to the half line
emanating from the origin and passing through $x$ for
each $x \in \partial D^p$, and that $f$ restricted to
$\partial F \times D^p$ is the projection to the second factor
multiplied by $1/2$,
where $D^p$ is the unit disk in $\R^p$ centered at the origin.
Such a manifold has a natural open book structure
with binding $\partial F \times \{0\}$, and the
resulting round fold map is said to be associated with the
open book structure and the Morse function $h$.
\end{ex}

Let $S^{n-1}$ be the unit sphere centered
at the origin in $\R^n$ and $\gamma : S^{n-1} \to S^{n-1}$
the orientation reversing diffeomorphism
defined by
$$\gamma(x_1, x_2, \ldots, x_n) = (-x_1, x_2, \ldots, x_n)$$
for $(x_1, x_2, \ldots, x_n) \in S^{n-1}$.
In the following, we denote the total space of the non-orientable 
$S^{n-1}$--bundle over $S^1$ with monodromy $\gamma$, 
$$[0, 1] \times S^{n-1}/(1, q) \sim (0, \gamma(q)),$$
by $S^1 \tilde{\times} S^{n-1}$.

Furthermore, we denote by $S^2 \tilde{\times} S^2$
the total space of the unique non-trivial $S^2$--bundle over $S^2$.

The main theorem of this paper is the following.

\begin{thm}\label{thm:main}
A closed connected $n$--dimensional manifold
with $n \geq 4$ admits a round fold map into
$\R^{n-1}$ if and only if it is diffeomorphic to
one of the following manifolds:
\begin{enumerate}
\item standard $n$--dimensional sphere $S^n$,
\item a connected sum of finite numbers of copies
of $S^1 \times S^{n-1}$ or $S^1 \tilde{\times} S^{n-1}$,
\item $S^{n-2} \times \Sigma$ for a closed connected
surface $\Sigma$,
\item $S^2 \tilde{\times} S^2$ for $n=4$.
\end{enumerate}
\end{thm}

\section{Proof of Theorem~\ref{thm:main}}\label{section3}

In this section, we prove Theorem~\ref{thm:main}.

\begin{proof}[Proof of Theorem~\textup{\ref{thm:main}}]
Let $f : M \to \R^{n-1}$ be a round fold map of
a closed connected $n$--dimensional manifold with
$n \geq 4$.
In the following, for $r > 0$, $C_r$ denotes 
the $(n-2)$--dimensional
sphere of radius $r$ centered at the origin in $\R^{n-1}$.
We may assume that 
$$f(S(f)) = \bigcup_{r=1}^s C_r$$
for some $s \geq 1$ by composing a diffeomorphism of $\R^{n-1}$
if necessary.
Set $K = f^{-1}(0)$, which is a closed submanifold
of dimension $1$ of $M$ if it is not empty.
Let $D$ be the closed $(n-1)$--dimensional disk centered 
at the origin with radius $1/2$. 
Then, $f^{-1}(D)$ is diffeomorphic to
$K \times D$, which can be identified with
a tubular neighborhood $N(K)$ of $K$ in $M$. Furthermore, the composition
$\rho = \pi \circ f : M \setminus \Int{N(K)} \to S^{n-2}$
is a submersion, where $\pi : \R^{n-1} \setminus \Int{D}
\to S^{n-2}$ is the standard radial projection and
$\rho|_{\partial N(K)} : \partial N(K) = K \times \partial D
\to S^{n-2}$ corresponds to the projection to the second factor
followed by a scalar multiplication. Hence, $\rho$
is a smooth fiber bundle.  
In other words, $M$ admits an open book structure
with binding $K$.
The fiber (or the page) is identified with $F = f^{-1}(J)$, where
$$J = [1/2, \infty) \times \{0\} \subset \R \times \R^{n-2} = \R^{n-1},$$
and it is a compact surface possibly with boundary.
As we are assuming that $M$ is connected and $n \geq 4$, so is $F$.
Note that $h = f|_{F} : F \to J$ is
a Morse function with exactly $s$ critical points.

Note that all these arguments work even when
$K = \emptyset$. In this case, $F$ is a closed connected
surface and $M$ is the total
space of an $F$--bundle over $S^{n-2}$.


In the following, $A$ denotes the annulus $S^1 \times
[-1, 1]$, and $P$ denotes the compact surface obtained from
the $2$--sphere by removing three open disks: 
in other words, $P$ is a \emph{pair of pants}.
Furthermore, $B$ denotes the compact surface
obtained from the M\"obius band with an open
disk removed.

Set $M_r = f^{-1}(C_{[r-(1/2), r+(1/2)]})$, $r = 1, 2, \ldots, s$,
and $M_0 = f^{-1}(C_{[0, 1/2]})$, where 
$$C_{[a, b]} = \{(x_1, x_2, \ldots, x_{n-1})
\in \R^{n-1}\,|\, a \leq \sqrt{x_1^2 + x_2^2 + \cdots + x_{n-1}^2} \leq b\}$$
for $0 \leq a < b$.
Note that for $r \geq 1$, 
the map $\rho|_{M_r} : M_r \to S^{n-2}$
is a smooth fiber bundle and the fiber is diffeomorphic
to $D^2$, $P$ or $B$ together with a finite number of
copies of $A$ (see \cite{Sa04}, for example).

Now, suppose that $n \geq 5$. Then
by \cite{EE1, EE2, ES}, the identity component of the
group of diffeomorphisms
of $D^2$, $P$, $B$ and $A$ all have vanishing homotopy
groups of dimension $n-3$. Therefore, the above
bundles are all trivial.
Furthermore, for obtaining $M \setminus \Int{N(K)}$,
we need to glue the pieces by bundle maps with
fiber a disjoint union of circles. Again, as the 
identity component of the group of diffeomorphisms of $S^1$
has vanishing homotopy group of dimension $n-2$, we see
that the fiber bundle
$\rho : M \setminus \Int{N(K)} \to S^{n-2}$
is trivial. Moreover, the diffeomorphism used
for attaching $N(K)$ to $M \setminus \Int{N(K)}$
is again standard. As a result we see that
$M$ is diffeomorphic to the union
$$(K \times D^{n-1}) \cup_\partial (F \times S^{n-2})
= (\partial F \times D^{n-1}) \cup_\partial
(F \times \partial D^{n-1}),$$
where the attaching diffeomorphism is the standard one.
This implies that $M$ is diffeomorphic to
$\partial (F \times D^{n-1})$.

If $F$ has no boundary, then $M$ is diffeomorphic to
$F \times S^{n-2}$, where $F$ is a closed connected surface.
If $F$ has non-empty boundary, then $F \times D^{n-1}$
is diffeomorphic to an $(n+1)$--dimensional
manifold obtained by attaching some $1$--handles to $D^{n+1}$.
Therefore, its boundary is diffeomorphic to the
connected sum of finite numbers of copies of $S^1
\times S^{n-1}$ or $S^1 \tilde{\times} S^{n-1}$.

Now suppose that $n=4$. In this case,
$P$--bundles and $B$--bundles over $S^2$ are all
trivial, while $D^2$--bundles and $A$--bundles
may possibly be non-trivial.
If $F$ is a closed connected surface, then
$M$ is diffeomorphic to the total space
of an $F$--bundle over $S^2$. 
If $F$ is diffeomorphic to $S^2$, then
we see that $M$ is diffeomorphic either to
$S^2 \times S^2$ or to $S^2 \tilde{\times} S^2$.
If $F$ is not diffeomorphic to $S^2$, then
a $P$--bundle piece or a $B$--bundle piece necessarily appears,
and such a bundle must be trivial.
Since the boundary $S^1$--bundle (or $S^1$--bundles)
of an arbitrary non-trivial $D^2$--bundle (resp.\ $A$--bundle)
over $S^2$ is (resp.\ are) always non-trivial, no such non-trivial
bundle appears. This implies that the $F$--bundle
over $S^2$ must be trivial.
Therefore, $M$
is diffeomorphic to $S^2 \times \Sigma$ for a closed
connected surface $\Sigma$.

If $F$ has non-empty boundary, then as $N(K)$
is a trivial bundle over $D$, the boundary of
$M \setminus \Int{N(K)}$
is a trivial $\partial F$--bundle over $S^2$.
Then by \cite{ES}, we see that the $F$--bundle
over $S^2$, $\rho : M \setminus \Int{N(K)} \to S^2$,
is trivial. Therefore, by an argument similar to the above,
we see that $M$ is diffeomorphic to the connected sum
of finite numbers of copies of $S^1 \times S^3$ and
$S^1 \tilde{\times} S^3$.

Conversely, suppose that $M$ is one of the manifolds
listed in the theorem. 

If $M$ is the standard $n$--dimensional sphere $S^n$,
then the standard projection $\R^{n+1} \to \R^{n-1}$
restricted to the unit sphere $S^n$ is a round fold map:
in fact, it is a so-called special generic map (see \cite{Sa93}, for example),
and has only definite fold as its singularities.

If $M$ is a connected sum of $a$ copies
of $S^1 \times S^{n-1}$ and $b$ copies of $S^1 \tilde{\times} S^{n-1}$,
then let us consider the compact surface with boundary, say $F$, obtained from
the $2$--disk by attaching $a$ orientable $1$--handles and $b$
non-orientable $1$--handles along the boundary. Then,
$(F \times S^{n-2}) \cup (\partial F \times D^{n-1})
= \partial (F \times D^{n-1})$ admits a round fold map into $\R^{n-1}$,
as is seen from the construction given in Example~\ref{ex1},
and this manifold is diffeomorphic to $M$.

If $M = S^{n-2} \times \Sigma$ for a closed connected
surface $\Sigma$, then it obviously admits a round fold map
into $\R^{n-1}$, as is seen by using the construction given
in Example~\ref{ex1} again.

Finally, if $n=4$ and $M = S^2 \tilde{\times} S^2$, then
it admits a special generic map into $\R^3$ which is also
a round fold map \cite{Sa93}.

This completes the proof.
\end{proof}

As a direct corollary to Theorem~\ref{thm:main}, we immediately get the following.

\begin{cor}
A closed $n$--dimensional manifold
with $n \geq 4$ homotopy equivalent to $S^n$
admits a round fold map into
$\R^{n-1}$ if and only if it is diffeomorphic to
the standard $n$--dimensional sphere.
\end{cor}

Now let us discuss directed round fold maps.
Let $f : M \to \R^{n-1}$ be a round fold of a closed
\emph{orientable} $n$--dimensional manifold $M$.
For a component $c$ of $f(S(f))$, take a small
arc $\alpha \cong [-1, 1]$ in $\R^{n-1}$ that intersects $f(S(f))$
exactly at one point in $c$ transversely. We also assume that
the point $\alpha \cap f(S(f))$ is not an end point of $\alpha$.
Then, $f^{-1}(\alpha)$ is a compact surface
with boundary $f^{-1}(a) \cup f^{-1}(b)$, which
is diffeomorphic to a finite disjoint union of circles, where
$a$ and $b$ are the end points of $\alpha$.
Furthermore,
$f|_{f^{-1}(\alpha)} : f^{-1}(\alpha) \to \alpha$
can be regarded as a Morse function with exactly
one critical point. As $M$ is orientable, we see that
$f^{-1}(\alpha)$ is diffeomorphic to the union of $D^2$
(or $P$) and a finite number of copies of $A$
(see \cite{Sa04}, for example).
Therefore, the number of components of $f^{-1}(a)$
differs from that of $f^{-1}(b)$ exactly by one. If $f^{-1}(a)$
has more components than $f^{-1}(b)$, then we normally
orient $c$ from $b$ to $a$: otherwise, we orient $c$ from $a$ to $b$.
It is easily shown that this normal orientation is 
independent of the choice of $\alpha$.
In this way, each component of $f(S(f))$ is normally
oriented. If the normal orientation points inward,
then the component is said to be \emph{inward-directed}:
otherwise, \emph{outward-directed}.

\begin{dfn}
Let $f : M \to \R^{n-1}$ be a round fold map
of a closed orientable $n$--dimensional manifold. 
We say that $f$ is \emph{directed}
if all the components of $f(S(f))$ are inward-directed.
It is easy to see that a round
fold map $f$ is directed if and only if the number of
components of a regular fiber over a point in the
innermost component of $\R^{n-1} \setminus f(S(f))$ coincides with
the number of components of $S(f)$.
\end{dfn}

Then, as a corollary of the above proof of Theorem~\ref{thm:main},
we have the following.

\begin{thm}\label{thm:directed}
A closed connected orientable $n$--dimensional manifold
with $n \geq 4$ admits a directed round fold map into
$\R^{n-1}$ if and only if it is diffeomorphic to
one of the following manifolds:
\begin{enumerate}
\item standard $n$--dimensional sphere $S^n$,
\item connected sum of a finite number of copies
of $S^1 \times S^{n-1}$.
\end{enumerate}
\end{thm}

\begin{proof}
Suppose that $f : M \to \R^{n-1}$ is a directed
round fold map of a closed connected orientable $n$--dimensional
manifold.
Then, the surface $F$ given in the proof of Theorem~\ref{thm:main}
must be a compact orientable surface 
with non-empty boundary of genus zero, since $f$ is directed.
Then, we see that $M$ must be diffeomorphic to $S^n$
or to the connected sum of a finite number of copies
of $S^1 \times S^{n-1}$. Conversely, we see that $S^n$ and the
connected sum of a finite number of copies of $S^1 \times S^{n-1}$
admit directed round fold maps by using the open book
construction described in Example~\ref{ex1} associated with an appropriate Morse
function on a compact orientable surface of non-empty boundary
of genus zero. (More precisely, we use a Morse function on
such a surface such that the number of components of level sets
increases as the level value decreases.)
This completes the proof.
\end{proof}

For example, for a closed connected orientable surface
$\Sigma$, the manifold $S^{n-2} \times \Sigma$, $n \geq 4$,
admits a round fold map into $\R^{n-1}$, but does
not admit a directed one.

\section{Remarks and related results}\label{section4}

\begin{rem}
The manifolds appearing in Theorem~\ref{thm:main}
are all null-cobordant. In particular, their
Stiefel--Whitney numbers all vanish, and their signatures all
vanish when the dimension is divisible by $4$ and the manifold
is orientable. Compare this with \cite[Proposition~3.12]{Sa92}.
\end{rem}

\begin{rem}
The fundamental groups of the manifolds appearing in 
Theorem~\ref{thm:main} are either trivial, free,
or a surface fundamental group.
\end{rem}

\begin{rem}
In \cite[Proposition~3.6]{Sa93-2}, it has been given
a topological characterization of simply connected closed $4$--dimensional
manifolds that admit simple fold maps into $\R^3$.
Compare this with our Theorem~\ref{thm:main}.
\end{rem}

\begin{rem}
We have seen that if a closed connected $n$--dimensional
manifold admits a round fold map into $\R^{n-1}$, then
the manifold admits an open book structure over $S^{n-2}$.

On the other hand, as our proof shows,
for $n \geq 4$, if a closed connected $n$--dimensional
manifold admits an open book structure over $S^{n-2}$ with
non-empty binding, then
it must be diffeomorphic to
one of the following manifolds:
\begin{enumerate}
\item standard $n$--dimensional sphere $S^n$,
\item a connected sum of finite numbers of copies
of $S^1 \times S^{n-1}$ or $S^1 \tilde{\times} S^{n-1}$.
\end{enumerate}
If the binding is empty, then the manifold
is the total space of a $\Sigma$--bundle over $S^{n-2}$
for some closed connected surface $\Sigma$.
\end{rem}

\begin{rem}
For $n=3$, a characterization
of closed \emph{orientable} $3$--dimensional manifolds
that admit round fold maps into $\R^2$ has been
obtained in \cite{KS}.
For non-orientable $3$--dimensional manifolds,
such a characterization has not been known,
as far as the authors know.

Those closed orientable $3$--dimensional manifolds which
admit directed round fold maps into $\R^2$ have also been
characterized in \cite{KS}.
\end{rem}

Let $f : M \to \R^{n-1}$ be a round
fold map of a closed connected $n$--dimensional manifold $M$.
We denote by $n_0(f)$ (resp.\ $n_1(f)$) the number
of connected components of $S(f)$ with absolute index
$2$ (resp.\ $1$).

\begin{prop}\label{prop:euler}
When $n = \dim{M}$ is even, the Euler characteristic of $M$
is equal to $2(n_0(f) - n_1(f))$.
\end{prop}

\begin{proof}
We may assume that $f(S(f))$ is a concentric
family of spheres in $\R^{n-1}$. Then, $\ell \circ f : M \to \R$
is a Morse function on $M$, where $\ell$ is a projection to a real line.
We see that the number of critical points of $\ell \circ f$ with
even indices equals to $2n_0(f)$ and that of odd indices equals to $2n_1(f)$.
Then, the result follows.
\end{proof}

In the above proof, we see that the indices of the critical
points of the Morse function $\ell \circ f$ are 
$0$, $1$, $2$, $n-2$, $n-1$ or $n$.
In particular, $b_3(M) = b_4(M) = \cdots = b_{n-3}(M) = 0$ when $n \geq 6$,
where $b_j(M)$ is the $j$--th Betti number of $M$ (for any 
coefficient).

\begin{rem}
Proposition~\ref{prop:euler}
is also a consequence of a theorem of Fukuda \cite{Fuk}.
\end{rem}


\begin{ex}
Let us consider an arbitrary $(n-2)$--dimensional
closed connected manifold $X$ which embeds into $\R^{n-1}$,
$n \geq 4$,
such that the closed connected $n$--dimensional manifold
$M = \Sigma \times X$ is not diffeomorphic to any manifold
listed in Theorem~\ref{thm:main}, where $\Sigma$ is a closed
connected surface not diffeomorphic to $S^2$.
(For example, consider $X = S^1 \times S^{n-3}$.)
Then, $M = \Sigma \times X$ admits a fold map $f : M \to \R^{n-1}$
such that $f|_{S(f)}$ is an embedding onto a union
of parallel copies of $X$ embedded in $\R^{n-1}$.
(For example, for an arbitrary Morse function $h : \Sigma
\to \R$, consider the composition
$$M = \Sigma \times X \spmapright{h \times \id_X} \R \times X
\hookrightarrow \R^{n-1},$$
where the last map is an embedding.)
However, $M$ does not admit a round fold map into
$\R^{n-1}$ according to our Theorem~\ref{thm:main}.

Recall that for $n=3$, if a closed orientable
$3$--dimensional manifold admits a simple fold map
into $\R^2$, then it also admits a round fold map into $\R^2$
as has been shown in \cite{KS}. The above example
shows that this is not the case for $n \geq 4$ in general.
\end{ex}

\section{Classification of round fold maps}\label{section5}

In this section, we consider the classification problem
of round fold maps of closed $n$--dimensional manifolds
into $\R^{n-1}$, $n \geq 4$.


We recall the following standard definition.

\begin{dfn}
Let $f_i : M_i \to N_i$ be smooth maps of smooth
manifolds, $i = 0, 1$. We say that $f_0$ and $f_1$
are \emph{$C^\infty$ $\mathcal{A}$--equivalent} if
there exist diffeomorphisms $\psi : M_0 \to M_1$
and $\Psi : N_0 \to N_1$ such that
$f_1 = \Psi \circ f_0 \circ \psi^{-1}$.

Furthermore, we say that $f_0$ and $f_1$ are
\emph{$C^\infty$ $\mathcal{R}$--equivalent} if
$N_0 = N_1$ and
there exists a diffeomorphism $\psi : M_0 \to M_1$
such that $f_1 = f_0 \circ \psi^{-1}$.
\end{dfn}

For our classification of round fold maps up to
$\mathcal{A}$--equivalence, we will need the
following.

\begin{dfn}
Let $f : M \to \R^{n-1}$ be a round fold map
of a closed connected $n$--dimensional manifold.
Then, there exists a diffeomorphism $\Phi : \R^{n-1}
\to \R^{n-1}$ such that the critical value set
of the round fold map $\Phi \circ f$ coincides with
$\cup_{r=1}^s C_r$, where $s$ is the number of
connected components of $S(f)$. For $J = [1/2, \infty)
\times \{0\} \subset \R \times \R^{n-2} = \R^{n-1}$,
we set $F = (\Phi \circ f)^{-1}(J)$, which is a compact surface
and is connected if $n \geq 4$.
Then the restriction $\Phi \circ f|_F : F \to J$
is a Morse function. We call $F$ a \emph{page} of $f$
and the function $\Phi \circ f|_F$ a \emph{page
Morse function} associated with $f$.
\end{dfn}

\begin{lem}\label{lem:page}
Let $f: M \to \R^{n-1}$ be a round fold map of a
closed connected $n$--dimensional manifold.
Then, the page Morse functions associated with $f$
are unique up to $C^\infty$ $\mathcal{R}$--equivalence.
\end{lem}

\begin{proof}
Let $\Phi_i : \R^{n-1} \to \R^{n-1}$ be diffeomorphisms such that
the critical value set
of the round fold map $\Phi_i \circ f$ coincides with
$\cup_{r=1}^s C_r$, where $s$ is the number of connected
components of $S(f)$, $i = 0, 1$.
Set $J_0 = J$, $J_1 = \Phi_0 \circ \Phi_1^{-1}(J_0)$,
$F_i = (\Phi_0 \circ f)^{-1}(J_i)$
and $h_i = \Phi_0 \circ f|_{F_i}$, $i = 0, 1$.
Note that $h_0 : F_0 \to J_0$ is a page Morse function associated with $f$,
while $F_1 = (\Phi_1 \circ f)^{-1}(J_0)$ is another page for $f$ and 
$\Phi_1 \circ \Phi_0^{-1} \circ h_1: F_1 \to J_0$ 
is another page Morse function associated with $f$.

By integrating a certain vector field tangent
to $\cup_{r=1}^s C_r$, we can
construct a smooth isotopy $H_t : \R^{n-1} \to \R^{n-1}$,
$t \in [0, 1]$, with the following properties:
\begin{eqnarray}
H_0 & = & \id_{\R^{n-1}}, \\
H_t(C_r) & = & C_r, \quad r = 1, 2, \ldots, s, \\
H_1(J_0) & = & J_1, \\
H_1|_{J_0} & = & \Phi_0 \circ \Phi_1^{-1}|_{J_0}. \label{eq1}
\end{eqnarray}
Then, by lifting the vector field generated by the
isotopy $\{H_t\}_{t \in [0, 1]}$ with respect to $\Phi_0 \circ f$,
we can construct a smooth isotopy $\varphi_t : M \to M$,
$t \in [0, 1]$, such that the following holds:
\begin{eqnarray}
\varphi_0 & = & \id_M, \\
\varphi_t(S(f)) & = & S(f), \\
\varphi_1(F_0) & = & F_1, \\
H_t \circ \Phi_0 \circ f & = & \Phi_0 \circ f \circ \varphi_t, \quad
t \in [0, 1]. \label{eq2}
\end{eqnarray}
Then, we see that
\begin{eqnarray*}
\Phi_0 \circ f \circ \varphi_1|_{F_0} 
& = & H_1 \circ \Phi_0 \circ f|_{F_0} \\
& = & \Phi_0 \circ \Phi_1^{-1} \circ \Phi_0 \circ f|_{F_0}
\end{eqnarray*}
by virtue of (\ref{eq2}) and (\ref{eq1}) above.
This implies that
$$\Phi_1 \circ f \circ \varphi_1|_{F_0} = \Phi_0 \circ f|_{F_0}.$$
Hence, the two page Morse functions associated with $f$
are $C^\infty$ $\mathcal{R}$--equivalent.
\end{proof}

Then, we have the following classification theorem.

\begin{thm}\label{thm:classification}
Let $f_i : M_i \to \R^{n-1}$ be round fold maps
of closed connected $n$--dimensional manifolds with $n \geq 5$,
$i = 0, 1$. Then $f_0$ and $f_1$ are $C^\infty$
$\mathcal{A}$--equivalent if and only if their page Morse functions
are $C^\infty$ $\mathcal{R}$--equivalent.
\end{thm}

\begin{proof}
Necessity follows easily from Lemma~\ref{lem:page}.

Conversely, suppose that the page Morse functions
of $f_i$, $i = 0, 1$, are $C^\infty$ $\mathcal{R}$--equivalent.
We may assume that $f_i(S(f_i))$, $i = 0, 1$, are of the form
$\cup_{r=1}^s C_r$.
Set $J = [1/2, \infty) \times \{0\}
\subset \R \times \R^{n-2}$, $F_i = f_i^{-1}(J)$ and
$h_i = f_i|_{F_i} : F_i \to J$, $i = 0, 1$.
By assumption, there exists a diffeomorphism $\varphi : F_0
\to F_1$ such that $f_0 = f_1 \circ \varphi$.

Let us consider the decomposition
\begin{equation}
\R^{n-1} = D \cup R \cup L,
\label{eq:decomposition}
\end{equation}
where $D$ is the closed disk in $\R^{n-1}$
centered at the origin with radius $1/2$,
\begin{eqnarray*}
R = ([0, \infty) \times \R^{n-2}) \setminus
\Int{D} & \mbox{and} &
L = ((-\infty, 0] \times \R^{n-2}) \setminus \Int{D} \\
& \subset & \R \times \R^{n-2} = \R^{n-1}.
\end{eqnarray*}
Let us also consider the associated decomposition
$$M = f_i^{-1}(D) \cup f_i^{-1}(R) \cup f_i^{-1}(L),$$
$i = 0, 1$.
As has been observed in the proof of Theorem~\textup{\ref{thm:main}},
$f_i^{-1}(D)$ is either empty or the total space
of a trivial fiber bundle over $D$, where the
fiber is a finite union of circles,
while $f_i^{-1}(R)$ and $f_i^{-1}(L)$ are the total
spaces of locally trivial fiber bundles over the $(n-2)$--dimensional disk,
$i = 0, 1$. 

Let $T_R : [1/2, \infty) \times D^{n-2} \to R$ (or
$T_L : [1/2, \infty) \times D^{n-2} \to L$)
be the diffeomorphism defined by 
$T_R(r, x) = rx$ (resp.\ $T_L(r, x) = -rx$), where
we identify $D^{n-2}$ with
$$\{(x_1, x_2, \ldots, x_{n-1}) \in \R^{n-1}\,|\, x_1^2 + x_2^2 + \cdots
+ x_{n-1}^2 = 1, x_1 \geq 0\} \subset S^{n-2}.$$
Then, the map 
$$
T_R^{-1} \circ f_i|_{f_i^{-1}(R)} :
f_i^{-1}(R) \to [1/2, \infty) \times D^{n-2}$$
can be identified with a family
of Morse functions parametrized by $D^{n-2}$, $i = 0, 1$.
More precisely, to each $x \in D^{n-2}$ is associate
the Morse function 
$$\eta_1 \circ T_R^{-1} \circ f_i|_{F_{x, i}} : F_{x, i} \to [1/2, \infty),$$
where $F_{x, i} = f_i^{-1}(T_R([1/2, \infty) \times \{x\}))$
and $\eta_1 : [1/2, \infty) \times D^{n-2} \to [1/2, \infty)$
is the projection to the first factor.
By using a vector field argument (see, for example,
\cite{BH, BJ}), we can construct a trivializing
diffeomorphism $\tilde{\varphi}_i : F_i \times D^{n-2}
\to f_i^{-1}(R)$ in such a way that
we have the commutative diagram
$$
\begin{CD}
f_i^{-1}(R) @< \tilde{\varphi}_i << F_i \times D^{n-2} @> \tilde{\eta}_{2, i} >> D^{n-2} \\
@VV f_i V @VV h_i \times \id_{D^{n-2}} V @VV \id_{D^{n-2}} V \\
R @< T_R << [1/2, \infty) \times D^{n-2} @> \eta_2 >> D^{n-2},
\end{CD}
$$
where $\eta_2 : [1/2, \infty) \times D^{n-2} \to D^{n-2}$
and $\tilde{\eta}_{2, i} : F_i \times D^{n-2} \to D^{n-2}$
are the projections to the second factors, $i = 0, 1$.
A similar argument applies to $f_i^{-1}(L)$ as well.

By our assumption together with the above observation,
we see that 
$$f_0 : f_0^{-1}(R) \to R
\quad \mbox{ and } \quad f_0 : f_0^{-1}(L) \to L$$ 
are $C^\infty$ $\mathcal{R}$--equivalent
to 
$$f_1 : f_1^{-1}(R) \to R
\quad \mbox{ and } \quad
f_1 : f_1^{-1}(L) \to L,$$
respectively.

Now, through the above trivializations,
the attaching diffeomorphisms
between $f_i^{-1}(R)$ and $f_i^{-1}(L)$, $i = 0, 1$,
are identified with families
of diffeomorphisms of compact surfaces that
preserve the page Morse functions,
parametrized by $\partial D^{n-2} = S^{n-3}$.
By a result of Maksymenko \cite{Mak1, Mak2},
the space of such diffeomorphisms
has vanishing $(n-3)$--th homotopy group, as we are
assuming $n \geq 5$.
This implies that, by changing the above trivializations
slightly near the attaching parts, we may assume that
the attaching diffeomorphisms for $f_0$ and $f_1$
coincide with each other. 
Therefore, we can construct a diffeomorphism
$f_0^{-1}(R \cup L) \to f_1^{-1}(R \cup L)$
that gives $\mathcal{R}$--equivalence
between $f_0$ and $f_1$ over $R \cup L$.

Now, as $f_0$ and $f_1$ are trivial fiber bundles
over $D$, it is easy to extend this diffeomorphism
to a diffeomorphism $M_0 \to M_1$ that
gives $\mathcal{R}$--equivalence
between $f_0$ and $f_1$ over $\R^{n-1}$.
This is because the group of diffeomorphisms of the
circle has trivial $(n-2)$--th homotopy group.
This completes the proof.
\end{proof}

\begin{rem}
As the above proof shows, if the given
round fold map $f$ is of the standard form
(i.e., if $f(S(f))$ is of the form
$\cup_{r=1}^s C_r$), then its $C^\infty$
$\mathcal{R}$--equivalence class is determined
by the $\mathcal{R}$--equivalence class of 
its page Morse function.
\end{rem}

\begin{rem}
As has been shown in \cite{Izar4} (see also \cite[proof of
Theorem~3.8]{BCN}), two Morse functions on a compact connected
surface are $C^\infty$ $\mathcal{A}$--equivalent
if and only if their associated functions on the Reeb graphs
(with orientation reversing information when the source surface
is non-orientable)
are topologically equivalent. A Reeb graph is the quotient space
of the source surface obtained by identifying each connected
component of level sets to a point, and such a space
is known to have the structure of a finite graph
(see \cite{R}).
In particular, such Morse functions
can be classified up to $C^\infty$ $\mathcal{A}$--equivalence
by using purely combinatorial objects. If we use appropriate
functions on Reeb graphs, classification up to
$C^\infty$ $\mathcal{R}$--equivalence is also possible.
\end{rem}

Now, let us consider the case where $n = 4$.
In this case, we have the following.

\begin{thm}\label{thm:classification2}
Let $f_i : M \to \R^3$ be round fold maps
of a closed connected $4$--dimensional manifold $M$,
$i = 0, 1$. Then $f_0$ and $f_1$ are $C^\infty$
$\mathcal{A}$--equivalent if and only if 
exactly one of the following holds.
\begin{enumerate}
\item Both of $f_0$ and $f_1$ have indefinite
fold points and the page Morse functions of
$f_0$ and $f_1$ are $\mathcal{R}$--equivalent.
\item Both of $f_0$ and $f_1$ have only definite
fold points as their singularities and $M$ is
diffeomorphic to $S^4$.
\item Both of $f_0$ and $f_1$ have only definite
fold points as their singularities, $M$ is
diffeomorphic to an $S^2$--bundle over $S^2$,
and the self-intersection numbers of the components
of $S(f_0)$ coincide with those of $S(f_1)$ up
to order, with respect to a fixed orientation of $M$.
\end{enumerate}
\end{thm}

\begin{rem}
In Theorem~\ref{thm:classification2} (3), each of $S(f_i)$,
$i = 0, 1$, consists exactly of
two components and their self-intersection numbers are
of the form $k_i, -k_i$ for some integer $k_i$.
Note that $M$ is diffeomorphic to $S^2 \times S^2$ if the self-intersection
numbers are even and to $S^2 \tilde{\times} S^2$ if the
self-intersection numbers are odd.
\end{rem}

\begin{proof}[Proof of Theorem~\textup{\ref{thm:classification2}}]
Necessity is clear.

As to sufficiency, when both of $f_0$ and $f_1$
have indefinite fold points, the proof is the same
as that of Theorem~\ref{thm:classification}. 
This is because the space of the Morse functions
preserving the page Morse function is contractible
in this case \cite{Mak1, Mak2}.

When both of $f_0$ and $f_1$ have only definite
fold as their singularities, then as the
page function, we have the following two possibilities:
one is the Morse function on $D^2$ with exactly one
critical point, which is the maximum point, and the
other is the Morse function on $S^2$ with
exactly two critical points, which are
the minimum and the maximum points.

In the former case, the identity component 
of the group of diffeomorphisms of $D^2$ preserving the Morse
function has the homotopy type of $S^1$ \cite{Mak1, Mak2}.
However, if we use an attaching diffeomorphism for
$f_i^{-1}(R)$ and $f_i^{-1}(L)$ corresponding to a
non-trivial element of its fundamental group, then
the boundary $3$--dimensional manifold is not
diffeomorphic to $S^2 \times S^1$, which is a contradiction
as $f_i^{-1}(D)$ is a trivial circle bundle over $D$, $i = 0, 1$.
(Here, we use the decomposition of $\R^3$ as in (\ref{eq:decomposition}).)
Therefore, in this case, $f_0$ and $f_1$ are necessarily
$\mathcal{A}$--equivalent and $M$ is diffeomorphic to $S^4$.

In the latter case, $f_i^{-1}(D)$ is empty
and $S(f_i)$ consists exactly of two components,
$i = 0, 1$. Let $S_i^+$ (or $S_i^-$)
denote the component of $S(f_i)$ which is mapped
by $f_i$ to the outer (resp.\ inner) component 
of $f_i(S(f_i))$ in $\R^3$.
The attaching diffeomorphism for
$f_i^{-1}(R)$ and $f_i^{-1}(L)$ corresponds to
an element of the fundamental group of the identity component
of the group of diffeomorphisms
of the circle, which is isomorphic to the infinite
cyclic group $\Z$. Let $\sigma_i \in \Z$ denote
the corresponding element, $i = 0, 1$. Here, we may assume that
the self-intersection number of $S_i^+$ in $M$
coincides with $\sigma_i$. Then, the
self-intersection number of $S_i^-$ in $M$
are equal to $-\sigma_i$, $i = 0, 1$.
Recall that by our assumption, we have $\sigma_0 = \pm \sigma_1$.
If $\sigma_0 = \sigma_1$, then
we can easily construct an orientation preserving self-diffeomorphism
of $M$ that gives the $\mathcal{R}$--equivalence
between $f_0$ and $f_1$. If $\sigma_0 = -\sigma_1$, then
we consider an orientation reversing diffeomorphism
of $S^2$ that preserves the page Morse function for $f_0$, and this
gives rise to an orientation reversing self-diffeomorphism of $M$ that 
preserves $f_0$. Then, using this diffeomorphism, we may assume that
$\sigma_0 = \sigma_1$.
Therefore, in this case as well,
we conclude that $f_0$ and $f_1$ are $\mathcal{R}$--equivalent.
This completes the proof.
\end{proof}


It would be a difficult but interesting problem
to classify round fold maps on closed orientable
$3$--dimensional manifolds into $\R^2$ (refer to \cite{KS})
up to $\mathcal{A}$--equivalence.
Note that for simple stable maps of closed orientable
$3$--dimensional manifolds into $\R^2$, a classification
result has been given in \cite{Sa1}: however, the classification
is up to a certain equivalence strictly weaker than 
the $\mathcal{A}$--equivalence, and the result
is given in terms of links in the $3$--dimensional manifolds.

\section*{Acknowledgment}\label{ack}
The authors would like to thank Noriyuki Hamada
for stimulating discussions.
This work was supported by JSPS KAKENHI Grant Number 
JP17H06128. 



\end{document}